\documentclass[11pt,twoside]{amsart}

\usepackage{amsmath}
\usepackage{amsthm}
\usepackage{amsfonts, amssymb}
\usepackage[all]{xy}
\usepackage{url}

\usepackage{latexsym}

\usepackage{graphicx}
\usepackage{graphics}
\usepackage[hang,footnotesize,bf]{caption}
\usepackage{float}
\usepackage{enumerate}
\usepackage{verbatim}

\theoremstyle{plain}
\newtheorem{lema}{Lemma}[section]
\newtheorem{prop}[lema]{Proposition}
\newtheorem{teo}[lema]{Theorem}

\newtheorem{coro}[lema]{Corollary}
\theoremstyle{definition}
\newtheorem{defi}[lema]{Definition}

\newtheorem{obs}[lema]{Remark}

\newcommand{\bd}{\partial}

\begin{document}

\title[A generalization of a result of Dong and Santos-Sturmfels]{A generalization of a result of Dong and Santos-Sturmfels on the Alexander dual of spheres and balls}

\author[N.A. Capitelli]{Nicolas Ariel Capitelli}
\author[E.G. Minian]{Elias Gabriel Minian}

\address{Departamento  de Matem\'atica-IMAS\\
 FCEyN, Universidad de Buenos Aires\\ Buenos
Aires, Argentina}

\email{ncapitel@dm.uba.ar} \email{gminian@dm.uba.ar}

\begin{abstract} We prove a generalization of a result by Dong and Santos-Sturmfels about the homotopy type of the Alexander dual of balls and spheres. Our results involve $NH$-manifolds, which were recently introduced as the non-homogeneous (or non-pure) counterpart of classical polyhedral manifolds. We show that the Alexander dual of an $NH$-ball is contractible and the Alexander dual of an $NH$-sphere is homotopy equivalent to a sphere. We also prove that $NH$-balls and $NH$-spheres arise naturally as the double duals of standard balls and spheres. \end{abstract}



\maketitle

\section{Introduction}

Let $K$ be a finite simplicial complex. Fix a ground set of vertices $V$ which contains the set $V_K$ of vertices of $K$, and let $\Delta$ denote the simplex spanned by $V$. The classical Alexander duality theorem admits a  combinatorial formulation in terms of a simplicial homotopy representative $K^{*_V}$ of $|\partial\Delta|-|K|$ called the Alexander dual of $K$. In this form the theorem asserts that $H_i(K)\simeq H^{n-i-3}(K^{*_V})$, where $n$ is the cardinal of $V$ and both homology and cohomology groups are reduced (see \cite[Theorem 71.1]{Mu} and \cite{BjTa}). In light of this result, it is natural to ask if the homotopy type of $K$ can also be deduced from $K^{*_V}$. Unfortunately this is generally not the case. There are canonical examples of contractible polyhedra and homotopy spheres whose Alexander duals are not respectively contractible or homotopy equivalent to spheres. Moreover, it can be shown that for any finitely presented group $G$ there is a finite simply connected complex $K$ such that $\pi_1(K^{*_V})=G$ (see \cite{MR}). In 2002, Dong \cite{Dong} proved that the Alexander dual of a simplicial sphere has again the homotopy type of a sphere. One year later Santos and Sturmfels \cite{SS} showed that the Alexander duals of simplicial balls are contractible spaces. Dong's approach relies mainly on convexity and Santos-Sturmfels' proof uses Dong's result on spheres. Both results evidence that a locally well-behaved structure on the complex forces homotopy stableness on its dual and one may ask whether other manifold-like constructions can hold similar properties.

The $NH$-manifolds are natural candidates for this. These complexes were recently introduced in \cite{CM} as a generalization of combinatorial manifolds to the non-ho\-mo\-ge\-neous setting (or, more precisely, to the non-necessarily homogeneous setting). The study of $NH$-manifolds was in part motivated by the theory of non-pure shellability due to Bj\"orner and Wachs \cite{BjWa}. It was also shown in \cite{CM} that they appear when investigating  Pachner moves between manifolds: if two polyhedral manifolds (with or without boundary) are PL-homeomorphic then they are related by a finite sequence of factorizations involving $NH$-manifolds. 

In this paper we prove a generalization of the results of Dong and Santos-Sturmfels to non-necessarily homogeneous balls and spheres. Concretely, we show that the Alexander dual of an $NH$-ball is a contractible space and the Alexander dual of an $NH$-sphere is homotopy equivalent to a sphere (Theorems \ref{teo main for NH-balls} and \ref{teorema esferas completo} below). These results extend considerably the previous ones. Our approach is based on the local structure of theses polyhedra and, as a by-product, we exhibit an alternative and simple proof of the original results which relies on the local nature of the manifolds, in contrast to the previous treatments.

The second aim of this article is to use the theory of $NH$-manifolds to characterize the (topological, simplicial) structure of the Alexander duals of balls and spheres in the following sense. Given a subspace $A$ of the $d$-sphere $S^d$,  since the complement $B=S^d-A$ is also a subspace of $S^{d'}$ for any $d'\geq d$, it is natural to study the relationship between $A$ and $S^{d'}-B$ (the complement of its complement in a sphere of higher dimension). In the combinatorial setting, this amounts to understand the \emph{double dual} $L=(K^{*_V})^{*_{V'}}$ where $V_K\subseteq V\subsetneq V'$.  We prove that the double duals of balls (resp. spheres) are $NH$-balls (resp. $NH$-spheres). 

The rest of the paper is organized as follows. In section two we recall the basic properties of classical combinatorial manifolds and $NH$-manifolds and prove a result on the existence of spines for $NH$-manifolds with boundary. This result is used in the proofs of the main theorems but it is also interesting in its own right.

In section three we compare the Alexander duals of a complex with respect to different ground sets of vertices and characterize the Alexander double duals of balls and spheres. As a corollary we show that $NH$-balls and $NH$-spheres appear as double duals of (classical) balls and spheres.

In the last section of the article we prove the generalization of Dong's and Santos-Sturmfels' results on the Alexander dual of spheres and balls to the non-homogeneous setting.

\section{Preliminaries}\label{Sec: Peliminaries}

\subsection{Basic notations} All complexes considered in this paper are finite. Given a set of vertices $V$, $|V|$ will denote its cardinality and $\Delta(V)$ the simplex spanned by the vertices of $V$. $\Delta^d=\Delta(\{0,\ldots,d\})$ will denote a generic $d$-simplex and $\partial
\Delta^d$ its boundary. The set of vertices of a complex $K$ will be denoted $V_K$ and we set $\Delta_K=\Delta(V_K)$. We write $\sigma<\tau$ when $\sigma$ is a face of $\tau$ and $\sigma\prec\tau$ when it is an immediate face. A simplex is \emph{maximal} or \emph{principal} in a complex $K$ if it is not a proper face of any other simplex of $K$. A \emph{ridge} of $K$ is an immediate face of a principal simplex. Two simplices $\sigma,\tau\in K$ are \emph{adjacent} if $\sigma\cap\tau$ is an immediate face of $\sigma$ or $\tau$.

As usual $K\ast L$ will denote the join of the complexes $K$ and $L$. By convention, if $\emptyset$ is the empty simplex and $\{\emptyset\}$ the complex containing only the empty simplex then $K\ast\{\emptyset\}=K$ and $K\ast\emptyset=\emptyset$. For $\sigma\in K$,  $lk(\sigma,K)=\{\tau\in K:\ \tau\cap\sigma=\emptyset,\ \tau\ast\sigma\in K\}$ denotes its \emph{link} and $st(\sigma,K)=\sigma\ast lk(\sigma,K)$ its (closed) \emph{star}. The union of two complexes $K, L$ will be denoted by $K+L$. A subcomplex $L\subset K$ is said to be \emph{top generated} if every principal simplex of $L$ is also principal in $K$.

We write $K\simeq_{PL}L$ when $K$ and $L$ are PL-homeomorphic; that is, whenever they have a subdivision in common. We shall frequently identify a complex $K$ with its geometric realization and we shall write $K\simeq L$ when $K$ is homotopy equivalent to $L$.

Given $t\geq 0$, $\Sigma^t K=\partial\Delta^t\ast K$ will denote the simplicial $t$-fold (unreduced) suspension of $K$.

A principal simplex $\tau\in K$ is \emph{collapsible} in $K$ if there is a ridge $\sigma\prec\tau$ which is not a face of any other simplex of $K$ (i.e. $\sigma$ is a free face). If $\tau$ is collapsible, the operation which transforms $K$ into $K-\{\tau,\sigma\}$ is called an \emph{elementary (simplicial) collapse} and is denoted by $K\searrow^e K-\{\tau,\sigma\}$. It is easy to see that $K-\{\tau,\sigma\}\subset K$ is a strong deformation retract. The inverse operation is called an \emph{elementary (simplicial) expansion}. If there is a sequence $K\searrow^e K_1\searrow^e\cdots\searrow^e L$ we say that $K$ \emph{collapses} to $L$ (or equivalently, $L$ \emph{expands} to $K$) and write  $K\searrow L$  or $L\nearrow K$. A complex $K$ is \emph{collapsible} if it has a subdivision which collapses to a single vertex.

\subsection{Combinatorial manifolds} We recall some basic definitions and properties of the classical theory of combinatorial manifolds. For a comprehensive exposition of the subject we refer the reader to \cite{Gla, Hud, Lic}.

A combinatorial $d$-ball is a complex which is PL-homeomorphic to $\Delta^d$. A combinatorial $d$-sphere is a complex PL-homeomorphic to $\partial\Delta^{d+1}$. By convention, $\{\emptyset\}=\bd{\Delta}^{0}$ is considered a sphere of dimension $-1$. A combinatorial $d$-manifold is a complex $M$ such that $lk(v,M)$ is a combinatorial ($d-1$)-ball or ($d-1$)-sphere for every $v\in V_M$. It is easy to verify that $d$-manifolds are homogeneous complexes of dimension $d$; that is, all of its principal simplices are $d$-dimensional. The link of any simplex in a manifold is also a ball or a sphere and the class of combinatorial manifolds is closed under PL-homeomorphisms. In particular, combinatorial balls and spheres are combinatorial manifolds. By a result of J.H.C. Whitehead, combinatorial $d$-balls are precisely the collapsible combinatorial $d$-manifolds (see \cite[Corollaries III.6 and III.17]{Gla}) and by a result of Newman, if $S$ is a combinatorial $d$-sphere containing a combinatorial $d$-ball $B$, then the closure $\overline{S-B}$ is a combinatorial $d$-ball (see \cite{Gla, Hud, Lic}).

The boundary  $\bd M$ of a combinatorial $d$-manifold $M$ can be regarded as the set of simplices whose links are combinatorial balls. This coincides with the usual definition of boundary for $d$-homogeneous complexes as the subcomplex generated by the mod $2$ sum of the ($d-1$)-simplices. It is easy to see that $\partial M$ is a ($d-1$)-combinatorial manifold without boundary.

A \emph{weak $d$-pseudomanifold} without boundary is a $d$-homogeneous simplicial complex $P$ satisfying that each ($d-1$)-simplex is contained in exactly two $d$-simplices. It is easy to see that in this case $lk(\sigma,P)$ is a weak ($d-\dim(\sigma)-1$)-pseudomanifold for every $\sigma\in P$ and that $H_d(P;\mathbb{Z}_2)\neq 0$, since the mod 2 sum of the $d$-simplices of $P$ is a generating cycle. A \emph{$d$-pseudomanifold} is a weak $d$-pseudomanifold with or without boundary (i.e. the $(d-1)$-simplices are contained in \emph{at most} two $d$-simplices) which is strongly connected; that is, any two $d$-simplices $\sigma,\tau$ can be connected by a sequence of $d$-simplices $\sigma=\eta_0,\ldots,\eta_k=\tau$ such that $\eta_i\cap\eta_{i+1}$ is ($d-1$)-dimensional for each $i=0,\ldots,k-1$ (i.e. $\eta_i$ and $\eta_{i+1}$ are adjacent). It is easy to see that a connected combinatorial $d$-manifold is a $d$-pseudomanifold.

\subsection{Non-homogeneous manifolds} $NH$-manifolds are the non-homogeneous versions of combinatorial manifolds and play a key role in this work. We give next a brief summary of the subject and refer the reader to \cite{CM} for a more detailed exposition.

$NH$-manifolds have a local structure consisting of Euclidean spaces of varying dimensions. In Figure \ref{fig:ejemplos_nh_variedades} we exhibit some examples of $NH$-manifolds.

\begin{defi} An \emph{$NH$-manifold} (resp. \emph{$NH$-ball}, \emph{$NH$-sphere}) of dimension $0$ is a manifold (resp. ball, sphere) of dimension $0$. An $NH$-sphere of dimension $-1$ is, by convention, the empty set. For $d\geq 1$, we define by induction

 \begin{itemize}
 \item An \emph{$NH$-manifold} of dimension $d$ is a complex $M$ of dimension $d$ such that $lk(v,M)$ is an $NH$-ball of dimension $0\leq k\leq d-1$ or an $NH$-sphere of dimension $-1\leq k\leq d-1$ for all $v\in V_M$.
 \item An \emph{$NH$-ball} of dimension $d$ is a collapsible $NH$-manifold of dimension $d$.
 \item An \emph{$NH$-sphere} of dimension $d$ and \emph{homotopy dimension} $k$ is an $NH$-manifold $S$ of dimension $d$ such that there exist a top generated $NH$-ball $B$ of dimension $d$ and a top generated combinatorial $k$-ball $L$ such that $B + L=S$ and $B\cap L=\bd{L}$. We say that  $S=B+L$ is a \emph{decomposition} of $S$ and write $\dim_h(S)$ for the homotopy dimension of $S$.\end{itemize}\end{defi}

\begin{figure}[h]

\centering

\includegraphics[width=6.00in,height=2.00in]{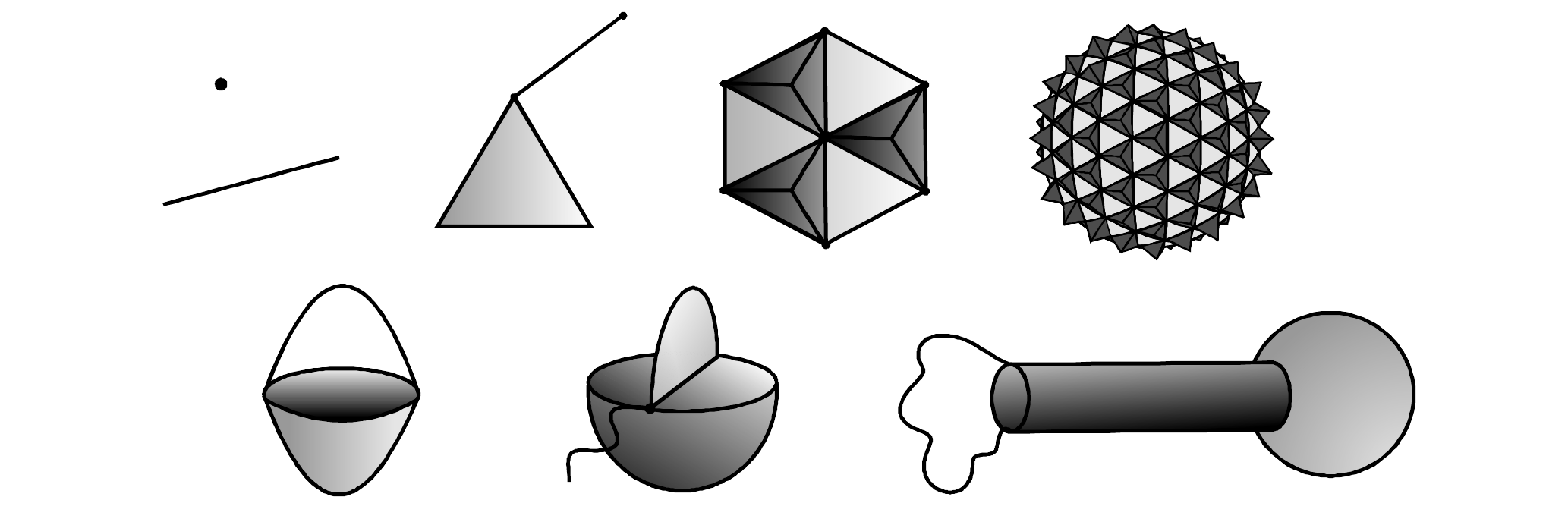}

\caption{Examples of $NH$-manifolds. The first, fourth and fifth figures are $NH$-spheres of dimension $1$, $3$ and $2$ and homotopy dimension $0$, $2$ and $1$ respectively. The second, third and sixth figures are $NH$-balls. The last figure is an $NH$-bouquet.}

\label{fig:ejemplos_nh_variedades}

\end{figure}

 In \cite{CM} it is proved that $NH$-manifolds satisfy many (generalized) results of the classical theory of combinatorial manifolds. Also, by \cite[Theorem 3.6]{CM}, homogeneous $NH$-manifolds are standard combinatorial manifolds. We next summarize the relevant results of this theory that will be used in this article.

\begin{teo} \label{teo recuento nh-variedades} Let $M$ be an $NH$-manifold of dimension $d$, let $\sigma\in M$ and let $B_1,B_2$ be $NH$-balls and $S_1,S_2$ be $NH$-spheres.\begin{enumerate}

\item \label{item link is bola o esfera} $lk(\sigma,M)$ is an $NH$-ball or an $NH$-sphere.

\item \label{item pl close} $NH$-manifolds, $NH$-balls and $NH$-spheres are closed under PL-homeomorphisms.

\item \label{item join de balls and spheres} $B_1\ast B_2$ and $B_1\ast S_2$ are $NH$-balls. $S_1\ast S_2$ is an $NH$-sphere.

\item \label{item nh-pseudo} If $M$ is connected then it is an \emph{$NH$-pseudomanifold}; i.e. ($a$) for each ridge $\sigma\in M$, $lk(\sigma,M)$ is either a point or an $NH$-sphere of homotopy dimension $0$; and ($b$) given any two principal simplices $\sigma,\tau\in M$, there exists a sequence $\sigma=\eta_1,\ldots,\eta_s=\tau$ of principal simplices of $M$ such that $\eta_i$ and $\eta_{i+1}$ are adjacent for each $1\leq i\leq s-1$.

\end{enumerate}

\end{teo}

The \emph{pseudoboundary} $\tilde{\partial}M$ of an $NH$-manifold is the set of simplices whose links are $NH$-balls (this is not in general a simplicial complex). The boundary $\partial M$ is the simplicial complex generated by the simplices in $\tilde{\partial}M$. By \cite[Proposition 4.3]{CM}, $\tilde{\partial}M$ is a complex if and only if $M$ is homogeneous. This implies that boundaryless $NH$-manifolds are classical (combinatorial) manifolds. This shows that, unlike classical manifolds, $NH$-spheres which are  non-homogeneous do have boundary. By \cite[Lemma 4.8]{CM} if $S=B+L$ is a decomposition of an $NH$-sphere then $lk(\sigma,S)$ is an $NH$-sphere with decomposition $lk(\sigma,B)+lk(\sigma,L)$ for every $\sigma\in L$.

Given an $NH$-manifold $M$, we denote by $\overset{\circ}{M}$ the relative interior of $M$, which is the set of simplices whose links are $NH$-spheres (of any dimension).

The following is a special case of \cite[Theorem 6.3]{CM}, which is a generalization of the classical Alexander's theorem on regular expansions (see \cite[Theorem 3.9]{Lic}).

\begin{teo} \label{Alexander for NH-manifolds} Let $M$ be an $NH$-ball (resp. $NH$-sphere) and $B$ a combinatorial ball. Suppose $M\cap B\subseteq\bd{B}$ is an $NH$-ball or an $NH$-sphere generated by ridges of $M$ or $B$ and that $(M\cap B)^{\circ}\subseteq\tilde{\partial} M$. Then \begin{enumerate}
\item $M+B$ is an $NH$-ball (resp. $NH$-sphere) if $M\cap B$ is an $NH$-ball.
\item $M+B$ is an $NH$-sphere if $M\cap B$ is an $NH$-sphere.
\end{enumerate}\end{teo}

The next two results, which are interesting in their own right,  will be used in the last section of the article. 

\begin{prop}		\label{prop d=k implies homogeneous}

	Let $M$ be a connected $NH$-manifold of dimension $d$ such that \linebreak $H_d(M;\mathbb{Z}_2)\neq 0$. Then, $M$ is a combinatorial $d$-manifold (without boundary). In particular, if $S$ is an $NH$-sphere with $\dim_h(S)=\dim(S)$ then $S$ is a combinatorial sphere.
	
	\begin{proof}  By \cite[Theorem 3.6]{CM} it suffices to prove that $M$ is homogeneous. Let $c$ be a generating $d$-cycle of $H_d(M;\mathbb{Z}_2)$ and let $K\subset M$ be the subcomplex generated by the $d$-simplices appearing in $c$ with nonzero coefficients. We shall show that $M=K$. Note that since $K\subset M$ is top generated and $M$ is an $NH$-pseudomanifold (see Theorem \ref{teo recuento nh-variedades} \eqref{item nh-pseudo}) then $K$ is a weak pseudomanifold without boundary (since $c$ is a cycle). If $M\neq K$, let $\eta\in M-K$ be a principal simplex adjacent to $K$ and set $\rho=\eta\cap K$. Since by dimensional considerations $\rho\prec\eta$ then $lk(\rho,M)=lk(\rho,M-\eta)+lk(\rho,\eta)$ is an $NH$-sphere of homotopy dimension $0$. But $lk(\rho,K)\subset lk(\rho,M-\eta)$ is a weak pseudomanifold without boundary and hence $H_{\dim(lk(\rho,K))}(lk(\rho,K);\mathbb{Z}_2)\neq 0$. This contradicts the fact that $lk(\rho,M-\eta)$ is an $NH$-ball since a generating cycle in $lk(\rho,K)$ is also generating in $lk(\rho,M-\eta)$. Note also that $\partial M=\partial K=\emptyset$.\end{proof}

\end{prop}

\begin{coro}[Existence of spines for $NH$-manifolds]		\label{coro existence of spines}

	Every connected $NH$-manifold $M$ with non-empty boundary has a spine (i.e. it collapses to a subcomplex of smaller dimension).

	\begin{proof} Let $d$ be the dimension of $M$ and let $Y^d$ be the $d$-homogeneous subcomplex of $M$ (i.e. the subcomplex of $M$ generated by the $d$-simplices). Start collapsing the $d$-simplices of $Y^d$ and suppose we get stuck before depleting all the $d$-simplices. Then, there is a boundaryless $d$-pseudomanifold $L\subset Y^d\subset M$ and hence $0\neq H_d(L;\mathbb{Z}_2)\subset H_d(M;\mathbb{Z}_2)$. By Proposition \ref{prop d=k implies homogeneous}, $M$ is a combinatorial manifold without boundary, which is a contradiction.\end{proof}

\end{coro}

\section{The Alexander double dual of balls and spheres} \label{Sec: Alexander Dual of Balls and Spheres}

\subsection{The Alexander dual with respect to different ground sets of vertices} We first study the relationship between the Alexander dual of a complex relative to its own set of vertices and to a bigger ground set of vertices. This is a natural question since geometrically it amounts to analyze the relation between the complement of a complex when seen as subspace of spheres of different dimensions.

 For a complex $K$ we denote $K^*=\{\sigma\in \Delta_K\,|\,\sigma^{c_{V_K}}\notin K\}$ the Alexander dual \emph{with respect to the ground set $V_K$}. Here $\sigma^{c_{V_K}}=\Delta(V_K-V_{\sigma})$ is the \emph{complement} of $\sigma$ in $V_K$. Now, for a vertex set $V\supseteq V_K$ we consider the simplex $\tau=\Delta(V-V_K)$ and we denote the Alexander dual of $K$ \emph{relative to} $V$ by $K^\tau$; that is, $K^{\tau}=\{\eta\in \Delta(V)\,|\,\eta^{c_V}\notin K\}$ where $\eta^{c_V}=\Delta(V-V_{\eta})$ is the complement of $\eta$ with respect to $V=V_K\cup V_{\tau}$. We will omit the subscript $V_K$ or $V$ in the complement when that is clear from the context. Note that if $\tau=\emptyset$ then $K^{\tau}=K^*$ is the Alexander dual of $K$ relative to its own set of vertices.

 We shall use the following convention regarding the Alexander dual of simplices and boundary of simplices: $(\Delta^d)^*=\emptyset$ and $(\partial\Delta^d)^*=\{\emptyset\}$. 
 
\begin{lema}		\label{lema formula central}
	Let $K$ be a simplicial complex and let $\tau$ be a (non-empty) simplex disjoint from $K$. Then,
	\begin{equation}	\tag{A}\label{eq formula central}
	K^{\tau}=\partial\tau\ast \Delta_K+\tau\ast K^*.\end{equation}
	Here $K^*$ is considered as a subcomplex of the simplex $\Delta_K$.
	
	In particular, we have the following consequences.
	\begin{enumerate}
		\item \label{item VK+Vtau} If $K$ is not a simplex or $\dim({\tau})\geq 1$ then $V_{K^{\tau}}=V_K\cup V_{\tau}$. If $K=\eta$ is a simplex and $\dim(\tau)=0$ then $\eta^{\tau}=\eta$. In any case, $V_{K}\subseteq V_{K^{\tau}}$.
		\item \label{item K tau dual = K} If $K$ is not a simplex or $\dim(\tau)\geq 1$ then $(K^{\tau})^*=K$.
		\item \label{item K dual rho = K} If $V_{K^*}\subsetneq V_K$ and $\rho=\Delta(V_K-V_{K^*})$ then $(K^*)^{\rho}=K$.
		\item \label{item K tau es suspension de K} If $K$ is not a simplex then $K^{\tau}\simeq\Sigma^t K^*$ for some $t\geq 0$.
	\end{enumerate}
		\begin{proof}
		Set $V=V_K\cup V_{\tau}$. Let $\sigma\in K^{\tau}$ be a principal simplex, so $\sigma^{c_V}\notin K$. If $\tau<\sigma$, say $\sigma=\tau\ast\eta$, then $\sigma^{c_V}=\eta^{c_{V_K}}$ and therefore $\sigma=\tau\ast\eta\in\tau\ast K^*$. Any other simplex in $K^{\tau}$ not containing $\tau$ lies trivially in $\partial\tau\ast \Delta_K$. For the other inclusion, if $\sigma=\tau\ast\eta$ is principal and $\eta\in K^*$ then $\sigma^{c_V}=\eta^{c_{V_K}}\notin K$, and hence $\sigma\in K^{\tau}$. If $\sigma\in\partial\tau\ast \Delta_K$ is principal then, in particular, $\Delta_K<\sigma$ and therefore $\sigma^{c_V}<\tau$. Since no vertex of $\tau$ lies in $K$, $\sigma^{c_V}\notin K$ and then $\sigma\in K^{\tau}$.
		
Item \eqref{item VK+Vtau} follows directly from formula \eqref{eq formula central} and items \eqref{item K tau dual = K}-\eqref{item K dual rho = K} from the fact that for a fixed ground set $V$, $(K^{*_V})^{*_V}=K$. Finally, \eqref{item K tau es suspension de K} follows from formula \eqref{eq formula central} since both summands are contractible (see Lemma \ref{lema dos contractiles y S1} \eqref{item K=A+B contractible}).		\end{proof}
\end{lema}

Note that the equation in (4) also holds for $\tau=\emptyset$ taking $t=0$. 

\begin{lema}\label{lema de cantidad de vertices} Let $K$ be a simplicial complex of dimension $d$ that is not a $d$-simplex. The following statements are equivalent.
\begin{enumerate}
\item $|V_K|=d+2$.
\item $V_{K^*}\neq V_K$.
\item $K\neq K^{**}$.
\end{enumerate}
\begin{proof}Suppose that $|V_K|=d+2$ and let $\sigma\in K$ be a $d$-simplex. Then the only vertex $v\in V_K-V_{\sigma}$ is not in $V_{K^*}$. Conversely, if $w\in V_K-V_{K^*}$ then $w^c\in K$. Since $K$ is not a $d$-simplex then $|V_K|\geq d+2$. Since $w^c$ is the simplex spanned by the vertices in $V_K-\{w\}$ and $\dim(K)=d$ then $|V_K|\leq d+2$. This proves that ($1$) and ($2$) are equivalent.

($2$) implies ($3$) since $V_{K^{**}}\subseteq V_{K^*}\subseteq V_K$. Also, ($3$) implies ($2$) since if $V_{K^*}=V_K$ then $K^{**}=K$.\end{proof}
\end{lema}

\begin{coro}		\label{coro predual minimal vertex}
	Let $K$ be a simplicial complex and let $\tau$ be a (non-empty) simplex disjoint from $K$. Then,
		\begin{enumerate}
			\item If $K$ is not a simplex or $\dim(\tau)\geq 1$ then $|V_{K^{\tau}}|=\dim(K^{\tau})+2$.
			\item The subcomplexes $\partial\tau\ast \Delta_K,\tau\ast K^*\subset K^{\tau}$ in formula \eqref{eq formula central} of Lemma \ref{lema formula central} are top generated.
		\end{enumerate}
	\begin{proof}
		If $K$ is not a simplex, item ($1$) follows directly from Lemmas \ref{lema formula central} and \ref{lema de cantidad de vertices}. If $K=\eta$ is a simplex and $\dim(\tau)\geq 1$ then $\eta^{\tau}=\partial\tau\ast\eta$ which has dimension $\dim(\tau)+\dim(\eta)$ and $\dim(\tau)+1+|V_{\eta}|=\dim(\tau)+\dim(\eta)+2$ vertices.
		
		For ($2$), simply notice that $\partial\tau\ast \Delta_K\cap \tau\ast K^*=\partial\tau\ast K^*$ and that $K^*$ is always properly contained in $\Delta_K$.\end{proof}
\end{coro}

\begin{obs}\label{remark K d d+2} Lemma \ref{lema formula central} and Corollary \ref{coro predual minimal vertex} state that every complex is the Alexander dual of a complex of dimension $d$ and $d+2$ vertices for some $d\geq 0$.\end{obs}

\subsection{Alexander double duals of balls and spheres} Suppose $A$ is a subspace of the $d$-sphere $S^d$. The complement $B=S^d-A$ is also a subspace of $S^{d'}$ for any $d'\geq d$ and taking into account that $S^d-B=A$ it is natural to ask what kind of relationship exists between $A$ and $S^{d'}-B$. In the simplicial setting this amounts to understand the similarities between a complex $K$ and $(K^{\tau})^{\sigma}$ for $V_{\tau}\cap V_K=\emptyset$ and $V_{\sigma}\cap V_{K^{\tau}}=\emptyset$. We call the complex $(K^{\tau})^{\sigma}$ a \emph{double dual} of $K$. When $\tau=\sigma=\emptyset$ we call $(K^*)^*=K^{**}$ the \emph{standard double dual} of $K$. 

Double duals share many of the properties of the original complexes. For example, it is easy to see from formula \eqref{eq formula central} of Lemma \ref{lema formula central} and Lemmas \ref{lema formula central} and \ref{lema de cantidad de vertices} that $(K^{\tau})^{\sigma}\simeq\Sigma^t K$ for some $t\geq 0$ if $|V_K|\geq d+3$. Also, it can be shown that a complex $K$ is shellable if and only if $(K^{\tau})^{\sigma}$ is shellable. We are mainly interested in double duals of combinatorial balls and spheres and we shall show that they are precisely the $NH$-balls and $NH$-spheres.
The result basically follows from the following

\begin{lema}
Let $K$ be a simplicial complex. If $V_K\subseteq V$ and $\eta\neq\emptyset$ is a simplex, then $$L=\partial\eta\ast\Delta(V)+\eta\ast K$$ is an $NH$-ball (resp. $NH$-sphere) if and only if $K$ is an $NH$-ball (resp. $NH$-sphere). Here $K$ is viewed as a subcomplex of the simplex $\Delta(V)$.

\end{lema}

\begin{proof}

 Put $\Delta=\Delta(V)$. If $L$ is an $NH$-ball or $NH$-sphere then $K=lk(\eta,L)$ is either an $NH$-ball or $NH$-sphere by Theorem \ref{teo recuento nh-variedades} \eqref{item link is bola o esfera}. Since $\partial\eta\ast\Delta$ and $\eta\ast K$ are collapsible and $\partial\eta\ast\Delta\cap\eta\ast K=\partial\eta\ast K$ then $K$ will be an $NH$-ball if $L$ is one and an $NH$-sphere if $L$ is one.

Suppose $K$ is an $NH$-ball or $NH$-sphere. By Theorem \ref{teo recuento nh-variedades} \eqref{item join de balls and spheres}, $\partial\eta\ast \Delta$ is a combinatorial ball, $\eta\ast K$ is an $NH$-ball and $\partial\eta\ast\Delta\cap\eta\ast K=\partial\eta\ast K$ is an $NH$-ball or $NH$-sphere according to $K$. We use Theorem \ref{Alexander for NH-manifolds} to prove that $L$ is an $NH$-ball or $NH$-sphere. Note that $\partial\eta\ast K$ is trivially contained in $\partial(\partial\eta\ast\Delta)$ and it is generated by ridges of $\eta\ast K$. Also, if $\rho\in(\partial\eta\ast K)^{\circ}$ and $\hat{\eta}$ denotes the barycenter of $\eta$ then $$lk(\rho,\eta\ast K)\simeq_{PL} lk(\rho,\hat{\eta}\ast\partial\eta\ast K)=\hat{\eta}\ast lk(\rho,\partial\eta\ast K)$$
which is an $NH$-ball by Theorem \ref{teo recuento nh-variedades} \eqref{item pl close}. This implies that $\partial\eta\ast K\subset\tilde{\partial}(\eta\ast K)$. By Theorem \ref{Alexander for NH-manifolds}, $L$ is an $NH$-ball or $NH$-sphere.
\end{proof}

\begin{teo}			\label{teorema principal 1 double dual}
	
	Let $K$ be a simplicial complex and let $\tau$ be a simplex (possibly empty) disjoint from $K$ and $\sigma$ a simplex (possibly empty) disjoint from $K^{\tau}$ . Then $K$ is an $NH$-ball (resp. $NH$-sphere) if and only if $(K^{\tau})^{\sigma}$ is an $NH$-ball (resp. $NH$-sphere).

\end{teo}

\begin{proof}

	 	We first prove the case $\tau=\sigma=\emptyset$. By Lemma \ref{lema de cantidad de vertices} we may assume $|V_K|=\dim(K)+2$. Let $\rho=\Delta(V_K-V_{K^*})\neq\emptyset$ so $K=(K^*)^{\rho}=\partial\rho\ast\Delta_{K^*}+\rho\ast K^{**}$ by Lemma \ref{lema formula central} \eqref{item K dual rho = K}. The result now follows from the previous lemma.
	 	
	 	If $K$ is a simplex and $\dim(\tau)=0$ the result is trivial. For the remaining cases we have	 	
	 	\begin{equation*} \label{eq 1 rumpel}
		(K^{\tau})^{\sigma}=\left\{\begin{array}{ll}
\partial\sigma\ast\Delta_{K^*}+\sigma\ast K^{**}& \tau=\emptyset\,,\,\sigma\neq\emptyset\\
K&\tau\neq\emptyset\,,\,\sigma=\emptyset\\
\partial\sigma\ast\Delta_{K^{\tau}}+\sigma\ast K& \tau\neq\emptyset\,,\,\sigma\neq\emptyset
\end{array}\right.
	\end{equation*}
	and the result follows from the previous lemma and the case $\tau,\sigma=\emptyset$.\end{proof}

\begin{coro} $NH$-balls are the double duals of combinatorial balls. $NH$-spheres are the double duals of combinatorial spheres.\end{coro}

It is known that a $d$-homogeneous complex with $d+2$ vertices is either the boundary of a simplex or an elementary starring of a simplex (see \cite[Lemma 6]{Mani}) but for a general $d$-dimensional complex with $d+2$ vertices not even its homotopy type can be known beforehand. However, when the complex is an $NH$-manifold then it is either contractible or homotopy equivalent to a sphere. Actually, the next stronger result holds.

\begin{prop}\label{propo d+2 nh manifolds} If $M$ is an $NH$-manifold of dimension $d$ and $d+2$ vertices then $M$ is an $NH$-ball or $NH$-sphere.\end{prop} 

\begin{proof} By Lemma \ref{lema de cantidad de vertices} $\rho=\Delta(V_M-V_{M^*})\neq\emptyset$ and  by Lemma \ref{lema formula central}  
$$M=(M^*)^{\rho}=\partial\rho\ast \Delta_{M^*}+\rho\ast M^{**}.$$
If $M^*$ is a simplex then $M=\partial\rho\ast\Delta_{M^*}$ is an $NH$-ball. Otherwise, since $M$ is an $NH$-manifold then $M^{**}=lk(\rho,M)$ must be an $NH$-ball or $NH$-sphere by Theorem \ref{teo recuento nh-variedades} \eqref{item link is bola o esfera}. Therefore $M$ is an $NH$-ball or an $NH$-sphere by Theorem \ref{teorema principal 1 double dual}.\end{proof}

\section{Main results}

In this section we generalize Dong's result on the Alexander dual of simplicial spheres \cite{Dong} and Santos-Sturmfels' result on simplicial balls \cite{SS} to the more general setting of $NH$-spheres and $NH$-balls. First we need some lemmas. For $v\in V_K$, $K-v=\overline{K-st(v,K)}$ denotes the \emph{deletion} of $v$.

\begin{lema}		\label{lema d+2 is the link}

Let $K$ be a complex of dimension $d$ and $d+2$ vertices. Then, for every vertex $u\in V_K-V_{K^*}$ we have that $K^*=(lk(u,K))^{\tau}$ where $\tau=\Delta(V_K-V_{st(u,K)}).$

\begin{proof} By hypothesis we can write $K=\Delta^d+u\ast lk(u,K)$. Let $\tau$ be as in the statement. Then,
\begin{align*}
\sigma\in (lk(u,K))^\tau & \Leftrightarrow  \Delta(V_{lk(u,K)}\cup V_{\tau}-V_{\sigma})\notin lk(u,K)\\
& \Leftrightarrow  \Delta(V_{lk(u,K)}\cup (V_K-V_{st(u,K)})-V_{\sigma})\notin lk(u,K)\\
& \Leftrightarrow  \Delta(V_K-\{u\}-V_{\sigma})\notin lk(u,K)\\
& \Leftrightarrow  u\ast\Delta(V_K-\{u\}-V_{\sigma})\notin K\\
& \Leftrightarrow  \Delta(V_K-V_{\sigma})\notin K\\
& \Leftrightarrow  \sigma\in K^*.\qedhere
\end{align*}\end{proof}

\end{lema}

\begin{lema}\label{lema link y deletion dual} Let $K\neq\Delta^d$ be a simplicial complex of dimension $d$ and let $v\in V_K$. Then,
\begin{enumerate}
\item \label{item la que vale siempre} $lk(v,K^*)=(K-v)^*$.
\item \label{item la que vale a veces} $lk(v,K)=(K^*-v)^{\tau}$ where $\tau=\Delta(V_{K-v}-V_{K^*-v})$.
\item \label{item B*-v suspension of link*} If $v$ is not isolated and $lk(v,K)$ is not a simplex then $K^*-v\simeq\Sigma^t lk(v,K)^*$ for some $t\geq 0$.
\item \label{item B*-v contractible} If $lk(v,K)$ is a simplex then $K^*-v$ is contractible.
\end{enumerate}
\begin{proof} For (1), $$\sigma\in lk(v,K^*)\Leftrightarrow v\ast\sigma\in K^* \Leftrightarrow(v\ast\sigma)^c\notin K\Leftrightarrow \sigma^c\notin K-v\Leftrightarrow \sigma\in (K-v)^*.$$

To prove (2), take any $x\notin V_K$. Since $K\neq\Delta^d$ then $(K^x)^*=K$ and by (1),  $$lk(v,K)=lk(v,(K^x)^*)=(K^x-v)^*.$$ Note that $K^x=\Delta_K+x\ast K^*$, and then $$K^x-v=\Delta_K-v+x\ast K^*-v=\Delta(V_K-v)+x\ast(K^*-v).$$ Now Lemma \ref{lema d+2 is the link} implies that $$(K^x-v)^*=lk(x,K^x-v)^{\tau}=(K^*-v)^{\tau}$$ where $\tau=\Delta(V_{K^x-v}-V_{st(x,K^x-v)})=\Delta(V_{K-v}-V_{K^*-v})$. This proves (2). 

To prove (3), apply Alexander dual to the equality given in (2) to yield $$lk(v,K)^*=((K^*-v)^{\tau})^*.$$ When $\tau\neq \emptyset$, this equals $K^*-v$ by Lemma \ref{lema formula central} \eqref{item K tau dual = K}, which settles the result with $t=0$. Note that, by hypothesis, $K^*-v=\Delta^r$ and $\dim(\tau)=0$ cannot simultaneously hold. 

Suppose now that $\tau=\emptyset$. Denote $T=K^*-v$. If $\dim(T)\neq |V_{T}|-2$ then $lk(v,K)^*=T^{**}=T$ by Lemma \ref{lema de cantidad de vertices} and the result holds with $t=0$. If $\dim(T)=|V_{T}|-2$ then $\rho=\Delta(V_{T}-V_{T^*})\neq\emptyset$ and $$T=(T^*)^{\rho}=\partial\rho\ast\Delta_{T^*}+\rho\ast T^{**}=\partial\rho\ast\Delta_{T^*}+\rho\ast lk(v,K)^*.$$ Since by hypothesis $\Delta_{T^*}=\Delta_{lk(v,K)}\neq\emptyset$ and $T^{**}=lk(v,K)^*\neq\emptyset$ then
$$K^*-v=T\simeq\Sigma(\partial\rho\ast lk(v,K)^*)\simeq\Sigma^t lk(v,K)^*.$$

To prove (4) note that if $(K^*-v)^{\tau}=lk(v,K)$ is a simplex then $K^*-v$ is an $NH$-ball by Theorem \ref{teorema principal 1 double dual}.\end{proof}

\end{lema}

The following result is standard.

\begin{lema}\label{lema dos contractiles y S1} Let $K$ be a finite simplicial complex and $A,B\subset K$ subcomplexes such that $K=A+B$.\begin{enumerate}
\item \label{item K=A+B contractible} If $A$ and $B$ are contractible  then $K\simeq \Sigma(A\cap B)$. If, in addition, $K$ is acyclic then $K$ is contractible. In particular, acyclic simplicial complexes of dimension $d$ and $d+2$ vertices are contractible.
\item \label{item AcapB y B contractible K=A} If $A\cap B$ and $B$ are contractible then $K\simeq A$.
\end{enumerate}
\end{lema}

The following is a rewriting of \cite[Lemma 2.4]{MR}.

\begin{lema}\label{remark MiRo}  Let $L$ be a subcomplex of $K$. Then $K\searrow L$ if and only if $K^*\nearrow L^{\tau}$ where $\tau=\Delta(V_K-V_L)$. In particular, if $L^*$ is contractible or homotopy equivalent to a sphere then so is $K^*$.\end{lema}

We are now able to give an alternative proof of Dong's and Santos-Sturmfels' original results.

\begin{teo}[Dong, Santos-Sturmfels]\label{teo dong} If $B\neq\Delta^d$ is a combinatorial $d$-ball then $B^{\tau}$ is contractible. If $S$ is a combinatorial $d$-sphere then $S^{\tau}$ is homotopy equivalent to a sphere.\end{teo}

\begin{proof} By Lemma \ref{lema formula central} \eqref{item K tau es suspension de K} it suffices to prove the result for $\tau=\emptyset$. We first prove it for a combinatorial ball $B$ by induction on $d\geq 1$. If $d=1$ then $B$ collapses to a $1$-ball with two edges (whose Alexander dual is a vertex) and the result follows from Lemma \ref{remark MiRo}. Now, let $d\geq 2$. If $|V_B|=d+2$, take $u\notin B^*$. If $lk(u,B)$ is not a simplex, Lemmas \ref{lema d+2 is the link} and \ref{lema formula central} \eqref{item K tau es suspension de K} imply $B^*\simeq\Sigma^t lk(u,B)^*$, which is contractible by induction since $lk(u,B)$ is a ball. If $lk(u,B)$ is a simplex, the result follows immediately. 

Suppose $|V_B|\geq d+3$ and let $v\in\partial B$. Now, $B^*-v$ is contractible by Lemma \ref{lema link y deletion dual} \eqref{item B*-v contractible} or Lemma \ref{lema link y deletion dual} \eqref{item B*-v suspension of link*} and induction. Since $B^*=B^*-v+st(v,B^*)$ is acyclic by Alexander duality then $B^*$ is contractible by Lemma \ref{lema dos contractiles y S1} \eqref{item K=A+B contractible}.

Now let $S$ be a combinatorial sphere. We may assume that $|V_S|\geq d+3$. We proceed again by induction on $d$. Let $d\geq 1$ and $v\in S$. By Lemma \ref{lema link y deletion dual} \eqref{item la que vale siempre}, $lk(v,S^*)=(S-v)^*$ which is contractible by Newman's theorem and the previous case. Since $S^*=S^*-v+st(v,S^*)$ where $(S^*-v)\cap st(v,S^*)=lk(v,S^*)$ is contractible, then $S\simeq S^*-v\simeq\Sigma^t lk(v,S)^*$ by Lemma \ref{lema dos contractiles y S1} \eqref{item AcapB y B contractible K=A} and Lemma \ref{lema link y deletion dual} \eqref{item B*-v suspension of link*}. The result now follows by the inductive hypothesis on the ($d-1$)-sphere $lk(v,S)$.\end{proof}

Note that this theorem actually holds for simplicial (not necessarily combinatorial) balls and spheres. The more general formulation follows from this result using an argument of Dong \cite{Dong}, since the non-trivial cases turn out to be polytopal which, in turn, are combinatorial (see also \cite{ES, Mani}).

In order to prove the generalization of Santos and Sturmfels' result, we first need to characterize the $d$-homogeneous subcomplex of an $NH$-ball of dimension $d$ and $d+3$ vertices. We need the following known result on manifolds with few vertices.

\begin{teo}[\cite{BK}, Theorem A]\label{kunel} Let $M$ be a boundaryless combinatorial $d$-manifold with $n$ vertices. If $$n<3\left\lceil\frac{d}{2}\right\rceil+3$$ then $M$ is a combinatorial $d$-sphere. Also, if $d=2$ and $n=6$ then $M$ is either PL-homeomorphic to a $2$-sphere or combinatorially equivalent to the projective plane $\mathbb{R}P^2$.\end{teo}

The following is an immediate consequence of this result.

\begin{coro} \label{kunelcoro}

 Let $M$ be a combinatorial $d$-manifold with boundary with $n$ vertices. If $$n<min\left\{3\left\lceil\frac{d-1}{2}\right\rceil+3,3\left\lceil\frac{d}{2}\right\rceil+2\right\}$$ then $M$ is a combinatorial $d$-ball. The result is also valid if $d=3$ and $n=6$.

\begin{proof} By Theorem \ref{kunel}  $\partial M$ is a combinatorial ($d-1$)-sphere. This  includes the case $d=3$ and $n=6$ since $\mathbb{R}P^2$ cannot be the boundary of a compact manifold. Take $u\notin M$  and build $N=M+u\ast\partial M$ where $M\cap u\ast\partial M=\partial M$. It is easy to see that $N$ is a boundaryless combinatorial $d$-manifold. Now, since $|V_N|<3\lceil\frac{d}{2}\rceil+3$ then $N$ is a combinatorial $d$-sphere by Theorem \ref{kunel} and $M=\overline{N-u\ast\partial M}$ is a combinatorial $d$-ball by Newman's theorem.\end{proof}

\end{coro}

\begin{prop}		\label{prop d-homogeneous part of balls and spheres}

	Let $B$ be an $NH$-ball of dimension $d$ and $n\leq d+3$ vertices. Then, the $d$-homogeneous subcomplex $Y^d\subset B$ is a combinatorial $d$-ball. 

	\begin{proof} Since $B$ is acyclic,  by Theorem \ref{teo recuento nh-variedades} \eqref{item nh-pseudo}  $Y^d$ is a weak $d$-pseudomanifold with boundary. We may assume $d\geq 2$ and $|V_{Y^d}|=d+3$ since the cases $d=0,1$ and $|V_{Y^d}|=d+1$ are trivial and,  if $|V_{Y^d}|=d+2$,  $Y^d$ is an elementary starring of a simplex by \cite[Lemma 6]{Mani}. Note that $Y^d$ is necessarily connected. We first prove that $Y^d$ is a combinatorial manifold. Let $v\in Y^d$. By the same reasoning as above we may assume $|V_{lk(v,B)}|=d+2$. If $lk(v,B)$ is an $NH$-ball then $lk(v,Y^d)$ is a combinatorial ($d-1$)-ball by inductive hypothesis since $lk(v,Y^d)$ is the ($d-1$)-homogeneous part of $lk(v,B)$.	Suppose $lk(v,B)$ is an $NH$-sphere. If $\dim_h(lk(v,B))=d-1$ then $lk(v,B)=lk(v,Y^d)$ is a combinatorial ($d-1$)-sphere by Proposition \ref{prop d=k implies homogeneous}. Otherwise, $lk(v,Y^d)$ is the ($d-1$)-homogeneous part of the $NH$-ball in any decomposition of $lk(v,B)$ and the result follows again by induction. This shows that $Y^d$ is a combinatorial $d$-manifold.

Suppose $d=2$. Note that $Y^d$ is $\mathbb{Z}_2$-acyclic since it is connected, it has non-empty boundary and it is contained in the acyclic complex $B$. On the other hand, any   $\mathbb{Z}_2$-acyclic complex with $5$ vertices is collapsible (see for example \cite[Theorem 1]{BD2}). 

For $d\geq 3$, $Y^d$ is a combinatorial $d$-ball by Corollary \ref{kunelcoro}.\end{proof}

\end{prop}

\begin{prop}\label{propoballd+3tod-2} Any $NH$-ball $B$ of dimension $d\geq 2$ and $d+3$ vertices collapses to a complex of dimension $d-2$.\end{prop}

\begin{proof} We first show that all the principal ($d-1$)-simplices in $B$ can be collapsed. Let $Y^{d-1}$ be the subcomplex of $B$ generated by the principal ($d-1$)-simplices and let $Y^d$ be the $d$-homogeneous part of $B$. By the previous proposition, $Y^d$ is a combinatorial ball. Suppose that not all the ($d-1$)-simplices in $Y^{d-1}$ can be collapsed. Let $K$ be the subcomplex of $Y^{d-1}$ generated by these ($d-1$)-simplices. By assumption, $K\neq \emptyset$. Note that $K$ is a weak ($d-1$)-pseudomanifold with boundary by Theorem \ref{teo recuento nh-variedades} \eqref{item nh-pseudo} but it has no free ($d-2$)-faces in $B$. Then $\partial K\subset Y^d$. Therefore, if $c$ denotes the formal sum of the ($d-1$)-simplices of $K$ then $c\in H_{d-1}(B,Y^d)$. Since $B$ and $Y^d$ are contractible then $H_{d-1}(B,Y^d)=0$. This implies that $c$ is not a generating cycle, which is a contradiction since the ($d-1$)-simplices of $c$ are maximal. This shows that we can collapse all the principal ($d-1$)-simplices in $B$. On the other hand, since $Y^d$ is a combinatorial $d$-ball with $d+3$ vertices or less, it is vertex decomposable by \cite[5.7]{KK}. In particular $Y^d$  is collapsible with no need of further subdivision. Then we can make the collapses in order of decreasing dimension and collapse the $d$-simplices and the ($d-1$)-simplices of $Y^d$ afterwards to obtain a ($d-2$)-dimensional complex.\end{proof}

\begin{coro}\label{corocollapsod-3} Any $NH$-ball of dimension $d\geq 3$ and $d+2$ vertices collapses to a complex of dimension $d-3$.\end{coro}

\begin{proof} We proceed by induction on $d$. If $d=3$ then $B$ is collapsible since it is acyclic and has few vertices (see \cite[Theorem 1]{BD2}). Let $d\geq 4$ and write $B=\Delta^d+st(u,B)$ where $u\notin\Delta^d$. Now, $\Delta^d\cap st(u,B)=lk(u,B)\subset\Delta^d$ is an $NH$-ball since $B$ is one. Also, $\dim(lk(u,B))\leq d-1$ and $|V_{lk(u,B)}|\leq d+1$. Let $m=|V_{lk(u,B)}|-\dim(lk(u,B))$. If $m=1$ then $lk(u,B)$ is a simplex and $B\searrow\Delta^d\searrow 0$. For $m=2,3,4$ we use the inductive hypothesis, Proposition \ref{propoballd+3tod-2} or Corollary \ref{coro existence of spines} respectively to show that $lk(u,B)$ collapses to a complex of dimension  $\dim(lk(u,B))-(5-m)=|V_{lk(u,B)}|-m-(5-m)=|V_{lk(u,B)}|-5\leq d+1-5=d-4$. Therefore, $u\ast lk(u,B)=st(u,B)$ collapses to a complex of dimension $d-3$. Finally, if $m\geq 5$ then $\dim(lk(u,B))\leq|V_{lk(u,B)}|-5\leq d-4$ and $\dim(st(u,B))\leq d-3$. In any case we can collapse afterwards the $i$-simplices of $\Delta^d$ ($i=d,d-1,d-2$) in order of decreasing dimension to obtain a ($d-3$)-dimensional complex.\end{proof}

We are ready to prove now the first of our main results.

\begin{teo}\label{teo main for NH-balls} Let $B$ be an $NH$-ball and let $\tau$ be a simplex (possibly empty). Then, $B^{\tau}$ is contractible.\end{teo}

\begin{proof}

 By Lemma \ref{lema formula central} \eqref{item K tau es suspension de K} we only need to prove the case $\tau=\emptyset$.  It suffices to prove that $B^*$ is simply connected. Let $d=\dim(B)$ and $n=|V_B|$. We can assume that $d\geq 2$ since in lower dimensions all $NH$-balls are combinatorial. We can also assume that 
  $2\leq n-d\leq 4$, since if $n-d\geq 5$, a simple argument of Dong \cite{Dong} shows that $B^*$ is simply connected (it contains the complete $2$-skeleton of $\Delta(V_{B^*})$). 
  
 If  $n\leq 7$, $B^*$ is collapsible since it is acyclic and it has few vertices (\cite[Theorem 1]{BD2}). For $n\geq 8$, by Proposition \ref{propoballd+3tod-2} and Corollaries \ref{coro existence of spines} and \ref{corocollapsod-3} there exists a subcomplex $K\subset B$ such that $B\searrow K$ with $V_K=V_B$ and $|V_K|-\dim(K)=5$. Therefore $B^*\nearrow K^*$, and since $|V_K|-\dim(K)=5$, $K^*$ is simply connected.\end{proof}

Our next goal is to prove the second of our main results.

\begin{teo}\label{teorema esferas completo} Let $S$ be an $NH$-sphere and let $\tau$ be a simplex (possibly empty). Then, $S^{\tau}$ is homotopy equivalent to a sphere.\end{teo}

Like in the proof for $NH$-balls, we only need to prove the case $\tau=\emptyset$ and $2\leq |V_S|-\dim(S)\leq 4$ since, as before, if $|V_S|-\dim(S)\geq 5$, then $S^*$ is simply connected, and a simply connected space with the homology of a sphere is homotopy equivalent to one. We can suppose also that $\dim_h(S)<\dim(S)$ by Proposition \ref{prop d=k implies homogeneous} and Theorem \ref{teo dong}. The $1$-dimensional case is easy to verify.

The proof of Theorem \ref{teorema esferas completo} will be divided in the following four cases. Let $d=\dim(S)\geq 2$, $n=|V_S|$ and  $k=\dim_h(S)$. We handle each case separately.

\begin{enumerate}[(A)]
\item\label{(A)} $n=d+2$ and $k=d-1$.
\item\label{(B)} $n=d+2$ and $k=d-2$.
\item\label{(C)} $n=d+3$ and $k=d-1$.
\item\label{(D)} Remaining cases.
\end{enumerate}

\begin{proof}[Proof of Case \eqref{(D)}]\label{coro main for nh-spheres 1} We will show that $S\searrow K$ with $|V_K|-\dim(K)=5$. The result will follow immediately from Lemma \ref{remark MiRo} and the fact that $K^*$ is simply connected. The case $n=d+4$ follows directly from Corollary \ref{coro existence of spines} by collapsing (only) the $d$-simplices of $S$.

Suppose now that $n=d+2$ or $d+3$ and let $S=B+L$ be a decomposition. We first analyze the case $n=5$. In this situation, $L=\ast$. If $d=2$ then $B$ is acyclic with four vertices and if $d=3$ then $B=\Delta^3$. Similarly as in the $1$-dimensional case, $S\searrow S^0$ and the result follows from Lemma \ref{remark MiRo}.

Suppose $n=d+3$ with $n\geq 6$. The complex $B$ in the decomposition of $S$ is an $NH$-ball of dimension $d$ and $|V_B|\in\{d+1,d+2,d+3\}$. In any case, $B$ collapses to a ($d-2$)-dimensional complex $T$ whether because $B=\Delta^d$ or by  Corollary \ref{corocollapsod-3} or Proposition \ref{propoballd+3tod-2}. Moreover, since $d\geq 3$, we can arrange the collapses in order of decreasing dimension to get $V_T=V_B$ by collapsing only the $d$ and ($d-1$)-dimensional simplices. Since $\dim(L)\leq d-2$ and it is top generated, the collapses in $B\searrow T$ can be carried out in $S$ and therefore $S\searrow K=T+L$, which is a complex with the desired properties.

The case $n=d+2$ with $n\geq 6$ follows similarly as the previous case by showing that $S$ collapses to a ($d-3$)-dimensional complex with the same vertices.\end{proof}

\begin{proof}[Proof of Case \eqref{(A)}]\label{coro d d+2 con d-1} We proceed by induction. Write $S=\Delta^d+u\ast lk(u,S)$ with $u\notin\Delta^d$. Note that $lk(u,S)$ is an $NH$-sphere of homotopy dimension $d-2$ and dimension $d-2$ or $d-1$. By Lemmas \ref{lema d+2 is the link} and \ref{lema formula central} \eqref{item K tau es suspension de K} it suffices to show that $lk(u,S)^*$ is homotopy equivalent to a sphere. If $\dim(lk(u,S))=d-2$ then $lk(u,S)$ is homogeneous by Proposition \ref{prop d=k implies homogeneous} and the result follows from Theorem \ref{teo dong}. If $\dim(lk(u,S))=d-1$ then $|V_{lk(u,S)}|=d+1$ and the result follows by the inductive hypothesis.\end{proof}

In order to prove the cases \eqref{(B)} and \eqref{(C)} we need some preliminary results.

\begin{lema}\label{lema S-v acyclic} Let $S=B+L$ be a decomposition of an $NH$-sphere. If $v\in L$ then $S-v$ is contractible. 
\end{lema}

\begin{proof}  If $v\in L^{\circ}$ then $L-v$ deformation retracts to $\partial L\subset B$ and, hence, $S-v\simeq B\simeq\ast$. Otherwise, $v\in\partial L\cap\tilde{\partial}B$ and $S-v=(B-v)+(L-v)$ with $(B-v)\cap(L-v)=\partial L-v$. Since $v\in\partial L\cap\tilde{\partial}B$, then $B-v$ and $L-v$ are contractible. On the other hand, $\partial L-v$ is contractible by Newman's theorem. Hence, $S-v$ is contractible.
\end{proof}

\begin{lema} \label{lema vertice esfera} Let $S=B+L$ be a decomposition of an $NH$-sphere of dimension $d\geq 1$ satisfying the hypotheses of case \eqref{(C)}. If $lk(v,S)$ is a combinatorial ($d-2$)-sphere then $S-v$ is an $NH$-ball.\end{lema}

\begin{proof} We proceed by induction in $d$. The case $d=1$ is straightforward. Let $d\geq 2$. We prove first that $S-v$ is an $NH$-manifold.

Let $w\in S-v$. We have to show that its link is an $NH$-sphere or an $NH$-ball.  If $w\notin st(v,S)$ then $lk(w,S-v)=lk(w,S)$ which is an $NH$-ball or $NH$-sphere. Suppose $w\in st(v,S)$. We will show first that $lk(w,S)$ is an $NH$-sphere of homotopy dimension $d-2$. We prove this in various steps. Note that this is clear if $w\in L$, so we may suppose $w\notin L$. 

{\it Step 1.} We first prove that if $v\notin L$ then there is a $d$-simplex in $st(w,S)$ which is adjacent to a $(d-1)$-simplex of $L$. Write $\Delta^d=\{v,w\}^c$. Since $v,w\notin L$ then $L\subset\Delta^d$ and therefore $\Delta^d\notin S$ because $L$ is top generated in $S$. Since $\dim(S)=d$ and $st(v,S)$ is ($d-1$)-homogeneous then $w$ is a face of $d$-simplex $\rho$ not containing $v$. Since any two ($d-1$)-faces of $\Delta^d$ are adjacent then $\rho$ is adjacent to some ($d-1$)-simplex of $L$.

{\it Step 2.} We now prove that the inclusion induces an isomorphism $H_{d-1}(S-w)\simeq H_{d-1}(B-w)$. On one hand, the induced homomorphism $H_{d-1}(B-w)\rightarrow H_{d-1}(S-w)$ is injective since $(S-w)-(B-w)=L-w$ is ($d-1$)-dimensional. To prove that it is also surjective we show that any ($d-1$)-cycle in $S-w$ cannot contain a ($d-1$)-simplex of $L$. Suppose $\sigma\in L$ is a non-trivial factor in a ($d-1$)-cycle $c^{d-1}$ of $S-w$. Then every ($d-1$)-simplex in $L$ appears in $c^{d-1}$ since $c^{d-1}$ is a cycle and $L$ is a top generated combinatorial ($d-1$)-ball. If $v\in L$ then every ($d-1$)-simplex of $st(v,S)$ appears in $c^{d-1}$ since $st(v,S)$ is also a top generated ($d-1$)-ball. In this case, at least one ($d-1$)-simplex of $st(v,S)$ belongs to $st(w,S)$, contradicting the fact that $c^{d-1}$ is a cycle in $S-w$. On the other hand, if $v\notin L$ then there exists by step 1 a principal ($d-1$)-simplex $\tau\in L$ with a boundary ($d-2$)-face $\eta<\rho\in st(w,S)$ with $\dim(\rho)=d$. Let $z=lk(\eta,\tau)$. Note that there are no $d$-simplices outside $st(w,S)$ containing $\eta$ since neither $v$, $w$ nor $z$ may belong to such $d$-simplex and $|V_S|=d+3$. Since $S$ is an $NH$-manifold, $\tau$ is the only principal ($d-1$)-simplex containing $\eta$, and then $\partial c^{d-1}\neq 0$ in $S-w$, which is a contradiction.

{\it Step 3.} We prove that $lk(w,S)$ is an $NH$-sphere of homotopy dimension $d-2$. We claim first that $H_{d-1}(S-w)=0$. By step 2 it suffices to show that $H_{d-1}(B-w)=0$. From the Mayer-Vietoris sequence applied to $B=B-w+st(w,S)$ and the fact that $lk(w,B)=lk(w,S)$ (here we use that $w\notin L$), it follows that $H_{d-1}(B-w)\simeq H_{d-1}(lk(w,S))$. If $H_{d-1}(lk(w,S))\neq 0$ then $lk(w,S)$ is ($d-1$)-homogeneous by Proposition \ref{prop d=k implies homogeneous}, which is a contradiction since $st(w,S)$ contains at least a ($d-1$)-simplex. Thus, the claim is proved.

If we now consider the Mayer-Vietoris sequence for $S=S-w+st(w,S)$ in degree $d-1$ one has that $\mathbb{Z}\simeq H_{d-1}(S)\rightarrow H_{d-2}(lk(w,S))$ is injective, so $H_{d-2}(lk(w,S))\neq 0$ and therefore $lk(w,S)$ is an $NH$-sphere of homotopy dimension $d-2$.

Finally if $\dim(lk(w,S))=d-2$ then $lk(w,S)$ is a combinatorial ($d-2$)-sphere by Proposition \ref{prop d=k implies homogeneous} and therefore, $lk(w,S-v)=lk(w,S)-v$ is a combinatorial ($d-2$)-ball by Newman's theorem. Suppose that $\dim(lk(w,S))=d-1$. If $|V_{lk(w,S)}|=d+1$ then we may write $lk(w,S)=\Delta^{d-1}+st(v,lk(w,S))$ since $v$ is not a vertex of a $d$-simplex in $S$. In this case $lk(w,S)-v=\Delta^{d-1}$. If $|V_{lk(w,S)}|=d+2$ we may apply the inductive hypothesis since $lk(v,lk(w,S))=lk(w,lk(v,S))$ is a combinatorial ($d-3$)-sphere, and conclude that $lk(w,S)-v$ is an $NH$-ball. This proves that $S-v$ is an $NH$-manifold. 

We prove now that $S-v$ is an $NH$-ball. Note that $\dim(S-v)=d$ and $|V_{S-v}|=d+2$, so  by Proposition \ref{propo d+2 nh manifolds} we only need to prove that it is acyclic, and this follows immediately from the Mayer-Vietoris sequence applied to $S=S-v+st(v,S)$.\end{proof}

\begin{lema}\label{prop v en L nonedge} Let $S=B+L$ be a decomposition of an $NH$-sphere satisfying the hypotheses of case \eqref{(C)}. If there is a vertex $v$ in $L$ such that $\dim(S-v)=d$ and there is a non-edge $\{u,w\}$ of $S$ with $u,w\neq v$ then $(S-v)^*$ is contractible.

\begin{proof} By Lemma \ref{lema de cantidad de vertices}, $|V_{S^*}|=d+3$. By hypothesis $\{u,w\}^c\in S^*$ is a $d$-simplex and since $v\neq u,w$ then $v\in\{u,w\}^c$. Therefore, $\dim(lk(v,S^*))=d-1$. On the other hand, there exists a $d$-simplex $\eta\in S$ with $v\notin\eta$; hence $\{v,a\}:=\eta^c\notin S^*$. Therefore, $|V_{lk(v,S^*)}|\leq d+1$. If $|V_{lk(v,S^*)}|=d$ then $(S-v)^*=lk(v,S^*)$ is a ($d-1$)-simplex. If $|V_{lk(v,S^*)}|=d+1$ then $lk(v,S^*)=(S-v)^*$ is acyclic by Lemma \ref{lema S-v acyclic} and Alexander duality, and therefore contractible by Lemma \ref{lema dos contractiles y S1} \eqref{item K=A+B contractible}.\end{proof}

\end{lema}

\begin{lema}		\label{prop (S-v)* contractil}

	Let $S$ be an $NH$-sphere satisfying the hypotheses of case \eqref{(C)}. Then, for any decomposition $S=B+L$ there exists $z\in V_L$ such that $(S-z)^*$ is contractible.

	\begin{proof} We proceed by induction in $d$. The $1$-dimensional case is straightforward. Let $d\geq 2$ and let $u\in V_L$. If $\dim(lk(u,S))=d-2$ then $lk(u,S)$ is a combinatorial $(d-2)$-sphere and the result follows from Lemma \ref{lema vertice esfera} and Theorem \ref{teo main for NH-balls}. Suppose $\dim(lk(u,S))=d-1$. We analyze the two possible cases $|V_{lk(u,S)}|=d+1$ or $|V_{lk(u,S)}|=d+2$.
	
If $|V_{lk(u,S)}|=d+1$, let $w\in S$ such that $\{u,w\}\notin S$. Let $\Delta^d$ be a $d$-simplex containing $u$ and let $v=V_S-V_{\Delta^d}-\{w\}$. Since $L$ is top generated then either $v\in L$ or $w\in L$. If $v\in L$ then Lemma \ref{prop v en L nonedge} implies that $(S-v)^*\simeq\ast$. Assume then that $v\notin L$ (and hence $w\in L$). We may assume $\dim(lk(w,S))=d-1$ since otherwise $w$ is the desired vertex by Lemma \ref{lema vertice esfera} and Theorem \ref{teo main for NH-balls} again. Let $\tilde{\Delta}^d$ be a $d$-simplex containing $w$. Since $L$ is top generated, $w\in L$ and $v\notin L$ then $\tilde{\Delta}^d=w\ast v\ast\Delta^{d-2}$ with $\Delta^{d-2}\prec\Delta^d-u$. Let $x\neq u$ be the only vertex in $\Delta^{d-2}$ not in $\tilde{\Delta}^d$. Then, $x\in L$ and it fulfils the hypotheses of Lemma \ref{prop v en L nonedge}. Therefore $(S-x)^*$ is contractible. 

Suppose finally that $|V_{lk(u,S)}|=d+2$. From the decomposition $lk(u,S)=lk(u,B)+lk(u,L)$ there exists $y\in lk(u,L)$ such that $(lk(u,S)-y)^*\simeq\ast$ by the inductive hypothesis. If $u\notin (S-y)^*$ then $u^c\in S-y$; i.e. $S-y-u=\Delta^d$. In this case, we can write $S-y=\Delta^d+u\ast lk(u,S-y)$ and we have $(S-y)^*=(lk(u,S)-y)^{\tau}$ by Lemma \ref{lema d+2 is the link}. If $lk(u,S)-y$ is not a simplex then $(lk(u,S)-y)^{\tau}\simeq \Sigma^t (lk(u,S)-y)^*\simeq\ast$ by Lemma \ref{lema formula central} \eqref{item K tau es suspension de K} and if $lk(u,S)-y=\Delta^r$ then $\tau\neq\emptyset$ and $(lk(u,S)-y)^{\tau}=\partial\tau\ast\Delta^r\simeq\ast$. In either case, $y$ is the desired vertex. Assume $u\in(S-y)^*$. Then we have a non-trivial decomposition $$(S-y)^*=(S-y)^*-u\underset{lk(u,(S-y)^*)}{+}st(u,(S-y)^*).$$ Since neither $S-y$ nor $lk(u,S)-y$ are simplices and $u\in S-y$ is not isolated then $(S-y)^*-u\simeq\Sigma^t lk(u,S-y)^*\simeq\ast$ by Lemma \ref{lema link y deletion dual} \eqref{item B*-v suspension of link*}. The result then follows by Lemmas \ref{lema dos contractiles y S1} \eqref{item K=A+B contractible} and \ref{lema S-v acyclic}.\end{proof}

\end{lema}

\begin{proof}[Proof of Cases \eqref{(B)} and \eqref{(C)}]\label{prop d-1,d+3 and d-2,d+2} We prove \eqref{(B)} and \eqref{(C)} together by induction in $d$. Let $S=B+L$ be a decomposition.

 If $d=2$, $B$ is collapsible since it is acyclic and has few vertices. Then $S\searrow S^0$ for \eqref{(B)} and $S\searrow S^1$ for \eqref{(C)}. The results then follow in both cases from Lemma \ref{remark MiRo}. 

Let $d\geq 3$. Suppose first that $S$ satisfies the hypotheses of \eqref{(B)}. Write $S=\Delta^d+v\ast lk(v,S)$. Then $S^*=lk(v,S)^{\tau}$ for $\tau=\Delta(V_S-V_{st(v,S)})$ by Lemma \ref{lema d+2 is the link}. Since $lk(v,S)$ is an $NH$-sphere of dimension $\leq d-1$ then the result follows from Theorem \ref{teo dong}, cases \eqref{(A)} and \eqref{(D)} or the inductive hypothesis on \eqref{(B)} and \eqref{(C)}.

Finally suppose $S$ satisfies the hypotheses of \eqref{(C)}. By Lemma \ref{prop (S-v)* contractil} there exists $v\in V_L$ such that $(S-v)^*\simeq\ast$. Write $S^*=S^*-v+st(v,S^*)$ where $(S^*-v)\cap st(v,S^*)=lk(v,S^*)=(S-v)^*\simeq\ast$. By Lemmas \ref{lema dos contractiles y S1} \eqref{item AcapB y B contractible K=A} and \ref{lema link y deletion dual}  \eqref{item B*-v suspension of link*}, $S^*\simeq S^*-v\simeq\Sigma^t lk(v,S)^*$  (note that $v$ is not isolated nor $lk(v,S)$ is a simplex because $v\in L$). Since $lk(v,S)$ is an $NH$-sphere of dimension $\leq d-1$ then $lk(v,S)^*$ is homotopy equivalent to a sphere by Theorem \ref{teo dong}, cases \eqref{(A)} and \eqref{(D)} or inductive hypothesis on \eqref{(B)} and \eqref{(C)}.\end{proof}

\end{document}